\documentclass{article}
\usepackage{amsmath,amsthm,amsfonts,enumerate}
\usepackage[all]{xy}
\newtheorem{theorem}{Theorem}[section]
\newtheorem{lemma}[theorem]{Lemma}
\newtheorem{proposition}[theorem]{Proposition}
\newtheorem{remark}[theorem]{Remark}
\newtheorem{corollary}[theorem]{Corollary}
\newtheorem{condition}[theorem]{Condition}

\newcommand{\C}{\mathbb{C}}

\newcommand{\Pt}{\textbf{Pt}}
\title{Nice exact categories are coexact}
\author{J. R. A. Gray}
\begin{document}
\maketitle
\begin{abstract}
Several important types of categories have been shown to be both
exact and coexact (in the sense of Barr). 
The first type consists of abelian categories, which due to their self-dual
definition, can be seen to be both exact and coexact 
by Tierney's characterization of them as additive exact categories. 
The next type consists of elementary toposes which are well-known to be exact, but
have also been shown to be coexact and coprotomodular by Bourn. 
In this paper we 
study a condition weaker than extensivity and equivalent to additivity for pointed categories. 
We show that for a finitely cocomplete category this condition together 
with exactness implies coexactness and coprotomodularity. As a special case
we obtain that a finitely cocomplete pretopos is coexact.

\end{abstract}
\section{Introduction}
This paper was in some sense inspired by the question, ``which
conditions together with exactness imply coexactness?''
%This paper is written in part as an invitation to further research
%into when are exact categories coexact.
The Tierney ``equation''
\[(\text{exact}) + (\text{additive})=(\text{abelian}) \]
immediately gives one answer to this question,
since the opposite category of an abelian category is abelian. One
of the aims of this paper, is to explain that extensivity is also an answer
to the question, subject to the existence of certain colimits, so that every
finitely cocomplete pretopos is coexact. In fact, we show that there
is a condition generalizing extensivity and equivalent to
additivity for pointed categories, which, together with exactness, implies coexactness,
provided certain colimits exist. In addition to this, we show
that every finitely cocomplete pretopos is co-ideally-exact, and co-arithmetical,
and that our above mentioned condition also implies coprotomodularity.
The former result recovers D. Bourn's results, showing that the dual
of a topos is exact, protomodular and arithmetical \cite{BOURN:2004b}. It
also recovers Theorem 1.3 of \cite{BORCEUX_CLEMENTINO:2025}, 
showing that the opposite category of pointed Hausdorff spaces is semi-abelian.

Let us begin by recalling the necessary background as well as
introducing the above-mentioned condition so that we can
explain more precisely what our aims are. 
A category $\C$ is regular \cite{BARR:1971} if it is finitely complete, admits
coequalizers of kernel pairs, and the class of regular epimorphisms is pullback
stable in it. A regular category $\C$ is called exact \cite{BARR:1971} when equivalence relations are effective,
that is, (the projections of) (internal) equivalence
relations are kernel pairs in it. A category $\C$ is extensive \cite{CARBONI_LACK_WALTERS:1993} if it admits finite coproducts and 
for each pair of objects $A$ and $B$ the coproduct functor
\begin{equation}
	\tag{D1}
	\label{extensive_equiv}
\vcenter{
\xymatrix@R=0.5ex{
	(\C\downarrow A)\times (\C\downarrow B) \ar[r] & (\C\downarrow A+B)\\
	((X,f),(Y,g)) \ar@{|->}[r]\ar[dddd]_{(u,v)} & (X+Y,f+g)\ar[dddd]^{u+v}\\
\\ \phantom{aaaa}\ar@{|->}[r] & \phantom{aaaa}\\ \\
	((X',f'),(Y',g')) \ar@{|->}[r] & (X'+Y',f+g')
}
}
\end{equation}
is an equivalence of categories. This turns out to be equivalent to 
$\C$ having finite coproducts, pullbacks of coproduct
	inclusions along arbitrary morphisms existing in $\C$, and for each diagram
\[
\xymatrix{
	A \ar[r]^{i}\ar[d]^{f} & C\ar@{}[dl]|{(3)}\ar@{}[dr]|{(4)}\ar[d]^{h} & B\ar[l]_{j}\ar[d]^{g} & \text{\scriptsize{(1)}}\\
A '\ar[r]_{i'} & C & B',\ar[l]^{j'} & \text{\scriptsize{(2)}}
}
\]
in $\C$,
where (2) is a coproduct, (1) is a coproduct if and only if (3) and (4) are pullbacks.
A category is a pretopos if it is exact and extensive. A category $\C$ is lextensive
if it is finitely complete and extensive.
A pointed category with binary products
is additive, when each hom-set is equipped with an abelian group structure, such that addition is 
bilinear with respect to composition. A category is abelian when it is pointed, admits finite products and coproducts, has
kernels and cokernels, and every monomorphism and epimorphism is normal.

Recall that for an object $B$, the category $\Pt_{\C}(B)$ has objects
triples $(A,\alpha,\beta)$ where $A$ is an object in $\C$ and $\alpha: A\to B$ and $\beta : B\to A$
are morphisms in $\C$ such that $\alpha \beta=1_B$. A morphism
$f:(A,\alpha,\beta)\to (A',\alpha',\beta')$ in $\Pt_{\C}(B)$ is a morphism $f:A\to A'$ such that 
$\alpha' f =  \alpha$ and $f\beta =\beta'$. Under the assumption that pullbacks of split
epimorphisms along arbitrary morphisms exist, each morphism $p:E\to B$ determines a pullback functor
$p^* : \Pt_{\C}(B) \to \Pt_{\C}(E)$. This functor sends an object $(A,\alpha,\beta)$ to 
$(E\times_B A,\pi_1,\langle 1_E,\beta p\rangle)$, where $(E\times_B A,\pi_1,\pi_2)$ is the pullback of $p$ 
and $\alpha$, and $\langle 1_E,\beta p\rangle$ is the unique morphism such that
$\pi_1\langle 1_E,\beta p\rangle=1_E$ and $\pi_2\langle 1_E,\beta p\rangle =\beta p$. The functor $p^*$ 
sends $f:(A,\alpha,\beta)\to (A',\alpha',\beta')$ to the unique morphism 
$p^*(f): (E\times_B A,\pi_1,\langle 1_E,\beta p\rangle) \to (E\times_B A',\pi'_1,\langle 1_E,\beta' p\rangle)$
such that $\pi'_1 p^*(f) = \pi_1$ and $\pi'_2 p^*(f)= f\pi_2$.
A category is protomodular \cite{BOURN:1991} if pullbacks of split epimorphisms along arbitrary morphisms
exist and each functor $p^*$ reflects isomorphisms. It is well-known that when a category $\C$ has
finite limits and an initial object, protomodularity is equivalent to requiring that for each morphism $p:0\to B$,
the functor $p^{*}$ reflects isomorphisms for monomorphisms.
A category is semi-abelian \cite{JANELIDZE_MARKI_THOLEN:2002} if it is pointed, exact, protomodular, and finitely cocomplete.
A category is ideally exact \cite{JANELIDZE:2024}, when it satisfies the same conditions as 
semi-abelian, except pointedness is weakened to the requiring that the unique morphism $0\to 1$ is a regular epimorphism.

For a category $\C$ with finite limits, we consider the following condition:
\begin{condition}
\label{conditions}
\begin{enumerate}[(a)]
\item finite coproducts exist in $\C$;
\item binary coproducts of pullback squares are pullback squares in $\C$;
\item for each monomorphism $m: S\to A$ in $\C$, the square
\[
\xymatrix{
S+S \ar[d]_{[1_S,1_S]} \ar[r]^{m+1_S} & A+S\ar[d]^{[1_A,m]}\\
S \ar[r]_{m} & A
}
\]
is a pullback;
\item for each regular epimorphism $p:E\to B$ in $\C$, the diagram
\[
\xymatrix{
E+E \ar[d]_{[1_E,1_E]} \ar[r]^{p+p} & B+B\ar[d]^{[1_B,1_B]}\\
E \ar[r]_{p} & B
}
\]
is a feeble pullback, by which we mean that the unique morphism into the pullback is
a regular epimorphism.
\end{enumerate}
\end{condition}
The main objectives of the paper are to prove the following:
\begin{enumerate}[(i)]
\item Condition \ref{conditions} 
holds for any lextensive category.
\item Condition \ref{conditions} 
is equivalent to additivity for pointed categories.
\item An exact category $\C$ satisfying Condition \ref{conditions}
has the following properties:
\begin{enumerate}
\item pushouts of monomorphisms exist in $\C$ and are pullbacks;
\item monomorphisms are stable under pushouts in $\C$;
\item monomorphisms are regular monomorphisms in $\C$ and hence $\C$ is balanced;
\item co-reflexive-relations in $\C$ are co-effective-equivalence-relations;
\item $\C$ is coprotomodular. 
\end{enumerate}
\item An exact category $\C$ satisfying Condition \ref{conditions}, admitting pushouts
of regular epimorphisms, is coexact.
\item A finitely cocomplete pretopos is co-ideally-exact and co-arithmetical.
\end{enumerate}

\section{Lextensive categories}
In this section, we prove that lextensive categories satisfy Condition \ref{conditions}.
We also show for a finitely complete category, that extensivity can be characterised 
via conditions closely related to Conditions \ref{conditions} (a)-(d)
(see Proposition \ref{extensivity:char} below).
\begin{lemma}\label{lemma:1}
For each morphism $f:A \to B$ in a lextensive category, the diagram
\begin{equation}
\tag{D2}
\label{codiagonal_diagram}
\vcenter{
\xymatrix{
A + A \ar[r]^-{[1_A,1_A]}\ar[d]_{f+f} & A\ar[d]^{f}\\
B+B\ar[r]_-{[1_B,1_B]} & B
}
}
\end{equation}
is a pullback.
\end{lemma}
\begin{proof}
Consider the diagram 
\[
\xymatrix{
	A \ar[d]_{f}\ar[r]_{i_s}\ar@/^3ex/[rr]^{1_A} & P \ar[r]^{q}\ar[d]_{p} & A\ar[d]^{f} \\
	B \ar[r]_{\iota_s} & B+B \ar[r]_{[1_B,1_B]} & B
}
\]
where the right hand square is a pullback, and $i_s$ is the unique morphism making the
	diagram commute ($s\in \{1,2\}$). By extensivity we know that $(P,i_1,i_2)$ is a coproduct, and hence,
since $p i_s=\iota_s f$ and $q i_s = 1_A$, the claim follows.
\end{proof}
Since the functors \eqref{extensive_equiv} are equivalences it follows that:
\begin{lemma}\label{lemma:2}
	Binary coproducts of pullbacks are pullbacks in an extensive category.\qed
\end{lemma}
\begin{lemma}\label{lemma:3}
	For morphisms $f:A\to X$, $g:B\to X$, $h:C\to X$ in a
	lextensive category, the diagram
	\[
		\xymatrix@C=10ex{
			P+R \ar[r]^-{[p_2,r_2]} \ar[d]_{p_1+r_1} & C\ar[d]^{h}\\
			A+B\ar[r]_{[f,g]} & X
		}
	\]
	is a pullback,
	where $(P,p_1,p_2)$ and  $(R,r_1,r_2)$ are
	pullbacks of $f$ and $h$, and $g$ and $h$,
	respectively.
\end{lemma}
\begin{proof}
	According to Lemmas \ref{lemma:1} and \ref{lemma:2}, each of the squares in the diagram
	\[
		\xymatrix{
			P+R \ar[r]^-{p_2+r_2}\ar[d]_{p_1+r_1}& C+C \ar[r]^-{[1_{C},1_{C}]}\ar[d]^{h+h} & C\ar[d]^{h}\\
			A+B \ar[r]_{f+g}&X+X \ar[r]_-{[1_X,1_X]} & X
		}
	\]
	is a pullback, and hence so is the desired diagram.
\end{proof}
As an immediate consequence we obtain:
\begin{lemma}\label{lemma:3.5}
For a monomorphism $m:S\to A$ in a
lextensive category, the diagram
\[
\xymatrix@C=10ex{
S+S \ar[r]^-{[1_S,1_S]} \ar[d]_{m+1_S} & S\ar[d]^{m}\\
A+S\ar[r]_{[1,m]} & A
}
\]
is a pullback.\qed
\end{lemma}
Combining Lemmas \ref{lemma:1}, \ref{lemma:2}, and \ref{lemma:3.5} we obtain: 
\begin{proposition}
Every lextensive category satisfies Condition \ref{conditions}.\qed 
\end{proposition}
In fact lextensivity is closely related to Condition \ref{conditions}:
\begin{proposition}\label{extensivity:char}
A finitely complete category $\C$ with initial object is extensive, if and only if
Condition \ref{conditions} (a) and (b) hold, and for each morphism $f:A\to B$ in $\C$, the
diagram \eqref{codiagonal_diagram} is a pullback.
\end{proposition}
\begin{proof}
Combining the previous results in this section, we see that it remains only to prove the ``if'' part of the statement.
If $i:I\to 0$  is a morphism in $\C$, then the diagram
\[
\xymatrix{
I+I\ar[r]^-{[1_I,1_I]} \ar[d]_{i+i}& I\ar[d]^{i}\\
0+0 \ar[r]_-{[1_0,1_0]} & 0
}
\]
is a pullback, and hence $[1_I,1_I]$ is an isomorphism. Therefore, the diagram
\[
\xymatrix{
I\ar[r]^{1_I} & I & I\ar[l]_{1_I}
}
\]
forms a coproduct, and hence there is at most one morphism from
$I$ to any object. However, since $I$ is trivially a weak initial
object, it must in fact be an initial object. It follows that $0$ is a strict
initial object. Now suppose $f:A\to A'$ and $g:B\to B'$ are morphisms
in $\C$.  Since the diagrams
\[
\xymatrix{
A\ar[d]_{f} \ar[r]^{1_A} & A\ar[d]^{f}\\
A'\ar[r]_{1_{A'}} & A'
}
\xymatrix{
0\ar[d] \ar[r] & B\ar[d]^{g}\\
0\ar[r] & B'
}
\]
are pullbacks, it follows that their coproduct
\[
\xymatrix{
A \ar[r]^-{\iota_1}\ar[d]_{f} & A+B \ar[d]^{f+g}\\
A' \ar[r]_-{\iota_1} & A'+B'
}
\]
is also a pullback.
Now suppose that the bottom row in the diagram
\[
\xymatrix{
	A \ar[r]^{i}\ar[d]^{f} & C\ar@{}[dl]|{(3)}\ar@{}[dr]|{(4)}\ar[d]^{h} & B\ar[l]_{j}\ar[d]^{g} & \text{\scriptsize{(1)}}\\
A '\ar[r]_{i'} & C & B'\ar[l]^{j'} & \text{\scriptsize{(2)}}
}
\]
is a coproduct, and that (3) and (4) are pullbacks. This means that the left hand square in the
diagram
\[
\xymatrix@C=8ex{
A+B \ar@/^3ex/[rr]^{[i,j]}\ar[r]_{i+j}\ar[d]_{f+g} & C+C\ar[d]^{h+h} \ar[r]_{[1_C,1_C]} & C\ar[d]^{h}\\
A'+B'\ar@/_3ex/[rr]_{[i',j']} \ar[r]^{i'+j'} & C'+C' \ar[r]^{[1_{C'},1_{C'}]} & C'
}
\]
is a pullback. Since the right hand square is also a pullback, it follows that the outer arrows form a pullback.
Noting that the $[i',j']$ is an isomorphism, we see that $[i,j]$ is too, and hence, (1) is a coproduct as desired.
\end{proof}
Recall that in a regular category with binary coproducts, the join of monomorphisms $s:S\to A$ and $t:T\to A$
can be constructed as the monomorphism forming part of the (regular-epi,mono)-factorization of
$[s,t]:S+T\to A$. Recall also that a regular category is coherent, if each $\textnormal{Sub}(A)$ has finite joins
which are preserved by each change-of-base morphism $f^{*}:\textnormal{Sub}(B)\to \textnormal{Sub}(A)$.
We end this section by recovering the well-known fact that every regular
extensive category is coherent, which in turn implies that each $\textnormal{Sub}(A)$ is a distributive lattice
(see \cite{JOHNSTONE:2002}).
\begin{proposition}\label{proposition:pretoposes_coherent}
Every regular extensive category is coherent.
\end{proposition}
\begin{proof}
Suppose $f:A\to B$ is a morphism in a regular extensive category, and $u:S\to B$ and $v:T\to B$
are monomorphisms. If 
\[
\xymatrix{
\bar S \ar[d]_{\bar u} \ar[r]^{h} & S\ar[d]^{u}\\
A\ar[r]_{f} & B &
}
\xymatrix{
\bar T \ar[d]_{\bar v} \ar[r]^{i} & T\ar[d]^{v}\\
A\ar[r]_{f} & B
}
\]
are pullbacks, then by Lemma \ref{lemma:3} we know that the outer arrows in the diagram
\[
\xymatrix@C=2ex@R=3ex{
\bar S + \bar T \ar[dd]_{[\bar u,\bar v]} \ar[rrrr]^{h+i}\ar[dr]^{\bar e} &&&& S+T\ar[dd]^{[u,v]}\ar[dl]_{e}\\
&\bar S\vee \bar T \ar[rr]^{f'}\ar[dl]^{\bar m} && S\vee T\ar[dr]_{m}&\\
A\ar[rrrr]_{f} &&&& B
}
\]
form a pullback. We complete diagram by taking (regular-epi,mono)-factorization of
the right hand vertical morphism and pulling it back to produce (by regularity) the 
(regular-epi,mono)-factorization of the left morphism.
The claim now follows from the above argument, and the fact that zero objects are strict
and hence preserved by $f^*$, in an extensive category.
\end{proof}
\section{Additive categories}
In this section we prove that a pointed category is additive if and only if it satisfies Condition \ref{conditions}.
Note that in this section, we will largely use $\oplus$ to denote coproduct (which is also product), and tend
to reserve $+$ and $-$
for addition and subtraction of morphisms. For objects $A$, $B$, $C$, $D$, and morphisms 
$f:A\to C$, $g:A\to D$, $h:B\to C$ and $i:B\to D$ in an additive category, we write (when convenient) 
$\tiny{\begin{bmatrix} f\\ g\end{bmatrix}}$ for the morphism $\langle f,g\rangle: A\to C\oplus D$ 
and $\tiny{\begin{bmatrix} f & h \\ g & i\end{bmatrix}}$ for the morphism
$\tiny{\begin{bmatrix}[f,h]\\ [g,i]\end{bmatrix}} = \tiny{[\begin{bmatrix}f\\g\end{bmatrix},\begin{bmatrix} h\\i\end{bmatrix}]}:A\oplus B\to C\oplus D$.
Note also that $f\oplus i = \tiny{\begin{bmatrix} f & 0 \\ 0 & i \end{bmatrix}}$.
\begin{lemma}\label{lemma:a1}
Let $\C$ be an additive category. For each monomorphism $m:S\to A$, the diagram 
\[
\xymatrix@C=10ex{
S\oplus S \ar[r]^-{[1_S,1_S]} \ar[d]_{m\oplus 1_S} & S\ar[d]^{m}\\
A\oplus S\ar[r]_{[1,m]} & A
}
\]
is a pullback.\end{lemma}
\begin{proof}
Suppose that $u:W\to S$ and $v=\tiny{\begin{bmatrix} v_1\\ v_2\end{bmatrix}} : W \to A\oplus S$ are morphisms
such that $[1,m]v=mu$. This means that $v_1 + mv_2=mu$ and hence
$v_1 = m(u-v_2)$. It follows that $(m\oplus 1_S)\tiny{\begin{bmatrix} u-v_2\\v_2\end{bmatrix}} =\tiny{\begin{bmatrix}m & 0 \\ 0 & 1_S\end{bmatrix}}\tiny{\begin{bmatrix} u-v_2\\v_2\end{bmatrix}}= \tiny{\begin{bmatrix} v_1\\v_2\end{bmatrix}}$
and $[1_S,1_S]\tiny{\begin{bmatrix} u-v_2\\v_2\end{bmatrix}}=u$. Uniqueness follows from the fact that $m \oplus 1_S$ is a monomorphism.
\end{proof}
\begin{lemma}\label{lemma:a2}
Let $\C$ be an additive category. For each regular epimorphism $p:E\to B$ in $\C$, the diagram 
\[
\xymatrix{
E \oplus E \ar[r]^-{[1_E,1_E]}\ar[d]_{p\oplus p} & E\ar[d]^{p}\\
B\oplus B\ar[r]_-{[1_B,1_B]} & B
}
\]
is a feeble pullback.
\end{lemma}
\begin{proof}
Since $\tiny{\begin{bmatrix} 1_B & -1_B \\ 0 & 1_B\end{bmatrix}}:B\oplus B\to B\oplus B$ is an isomorphism, it easily follows that the diagram 
\[
\xymatrix{
	E \oplus B \ar[r]^-{[1_E,0]}\ar[d]_{\tiny{\begin{bmatrix} 1_B & -1_B \\ 0 & 1_B\end{bmatrix}}(p\oplus 1_E)=\tiny\begin{bmatrix} p & -1_E \\ 0 & 1_E\end{bmatrix}} & E\ar[d]^{p}\\
B\oplus B\ar[r]_-{[1_B,1_B]} & B
}
\]
is a pullback.  The claim now follows, since ${\tiny (1_E \oplus p)\begin{bmatrix} 1_E & 1_E \\ 0 & 1_E\end{bmatrix}}$ is the desired induced morphism (which is a regular epimorphism since it is the composite of a regular epimorphism and an
isomorphism).
\end{proof}
Since binary coproduct and product coincide, and the binary product of two pullbacks is necessarily
a pullback, it follows that:
\begin{lemma}\label{lemma:a3}
Binary coproducts of pullbacks are pullbacks in an additive category.\qed
\end{lemma}
\begin{proposition}
A pointed category $\C$ is additive if and only if it satisfies Condition \ref{conditions}.\qed
\end{proposition}
\begin{proof}
Combining Lemmas \ref{lemma:a1}, \ref{lemma:a2} and \ref{lemma:a3}, we see that it only remains to prove the ``if'' part.
Note that we will use $+$ to denote coproduct until we establish that we have biproducts.
Since the diagram on the left and in the middle of Figure \ref{fig} are pullbacks, it follows that the diagram on the right is also a pullback.
\begin{figure}[htbp]
    \centering
\[
\xymatrix{
A \ar[r]^{1_A}\ar[d] & A\ar[d]&\\
0 \ar[r] & 0
}
\xymatrix{
B\ar[d]_{1_B}\ar[r] & 0\ar[d]&\\
B \ar[r] & 0
}
\xymatrix{
A+B \ar[r]^-{[1_A,0]}\ar[d]_{[0,1_B]} & A\ar[d]\\
B \ar[r] & 0
}
\]
    \caption{Construction of biproducts}
    \label{fig}
\end{figure}
It therefore follows that $\C$ admits biproducts. To prove that $\C$ is additive,
it is sufficient to prove that each morphism 
$m=\langle [1,1],\pi_2\rangle : A\oplus A \to A\oplus A$
is an isomorphism. We begin by showing that $m$ is a monomorphism. According to
Condition \ref{conditions} (c), the rectangle in the diagram
\[
\xymatrix@R=3ex@C=2ex{
A\oplus A \ar[rr]^{[1,1]}\ar@{-->}^{u}[dr]\ar[dd]_{\langle 1,1\rangle\oplus 1} &&A\ar[dd]^{\langle 1,1\rangle}\\
& R \ar[ur]^{[1,1]r_1}\ar[dd]_(0.3){\langle r_1,r_2\rangle} & \\
(A\oplus A)\oplus A \ar[rr]|(0.54){\hole}^(0.7){[1,\langle 1,1\rangle]}\ar[dr]_(0.4){\langle \pi_1\oplus 1,\pi_2\oplus 1\rangle\phantom{as}}&& A\oplus A\\
& (A\oplus A)\oplus (A\oplus A) \ar[ur]_(0.6){[1,1]\oplus [1,1]}
}
\]
a pullback. Writing $(R,r_1,r_2)$ for the kernel pair of $[1,1]:A\oplus A \to A$,
it follows that right hand parallelogram in the diagram above is a pullback, and hence
there is a unique morphism $u$ making the left hand parallelogram into a pullback.
Since $((A\oplus A)\oplus A,\pi_1\oplus 1,\pi_2\oplus 1)$ is the kernel pair of $\pi_2: A\oplus A\to A$,
and the left hand parallelogram is a pullback, it follows that the kernel pair 
of $m : A\oplus A \to A\oplus A$ is 
$(A\oplus A,1_{A+A},1_{A+A})$.  Applying Condition \ref{conditions} (c) again, we know that
the rectangle in the diagram
\[
\xymatrix{
A\oplus A \ar@{-->}[dr]^{v} \ar@/^3ex/[drr]^{1_{A\oplus A}}\ar@/_3ex/[ddr]_{\langle 1, \tiny{\begin{bmatrix} 0 & 1 \\ 0 & 0\end{bmatrix}}\rangle} \\
& (A\oplus A) \oplus (A\oplus A) \ar[r]^-{[1,1]}\ar[d]_{m\oplus 1} & A\oplus A\ar[d]^{m}\\
& (A\oplus A) \oplus (A\oplus A) \ar[r]_-{[1,m]} & A\oplus A
}
\]
is a pullback. Therefore, since the outer arrows commutate, there exists 
a unique morphism $v$ making the diagram commute.
It follows  $1_{A\oplus A} = \pi_1 (m\oplus 1) v= m\pi_1v$ and hence $m$ is
an isomorphism as desired.
\end{proof}
\begin{remark}
Note that in ``if'' part of the above proof we didn't use Condition \ref{conditions} (d),
so that in the pointed context it is redundant.
\end{remark}
\section{Condition \ref{conditions} in the exact context}
In this section, we prove that coexactness and coprotomodularity are consequences
of Condition \ref{conditions} together with exactness. As a corollary, we
show that the opposite category of a cocomplete pretopos, is arithmetical and ideally exact.
We start with a few facts about feeble pullbacks.
As mentioned above in a special case, by a feeble pullback in category with finite limits, we mean a commutative square
\begin{equation}
\label{feeble_pullback_1}
\tag{D3}
\vcenter{
\xymatrix{
A\ar[d]_{\alpha} \ar[r]^{f} & B \ar[d]^{\beta}\\
A'\ar[r]_{f'} & B' 
}
}
\end{equation}
such that the induced morphism $A\to A'\times_{\langle f',\beta \rangle} B$
into the pullback is a regular epimorphism.
\begin{lemma}\label{feeble_is_pullback}
If \eqref{feeble_pullback_1}
is a feeble pullback in a finitely complete category and $\alpha$ and $f$ are jointly monomorphic, then
it is a pullback.
\end{lemma}
\begin{proof}
Just note that $\alpha$ and $f$ being jointly monomorphic force the induced morphism into the
pullback to be a monomorphism, while being a feeble pullback forces it to be a regular epimorphism.
\end{proof}
\begin{lemma}\label{feeble_composite}
For a diagram
\[
\xymatrix{
A\ar[d]_{\alpha} \ar[r]^{f}\ar@{}[dr]|{\tiny\fbox{1}} & B \ar[d]^{\beta} \ar[r]^{g}\ar@{}[dr]|{\tiny\fbox{2}} & C \ar[d]^{\gamma}\\
A'\ar[r]_{f'} & B' \ar[r]_{g'} & C' 
}
\]
in a regular category $\C$, we have:
\begin{enumerate}[(a)]
\item if $\fbox{1}$ and $\fbox{2}$ are feeble pullbacks, then so is $\fbox{1}\fbox{2}$; 
\item if $\fbox{2}$ is a pullback, then $\fbox{1}$ is a feeble pullback if and only if $\fbox{1}\fbox{2}$ is a feeble pullback; 
\item if $\fbox{1}\fbox{2}$ is a feeble pullback and $f'$ is a 
regular epimorphism, then so is $\fbox{2}$; 
\end{enumerate}
\end{lemma}
\begin{proof}
Consider the diagram
\[
\xymatrix@C=2ex@R=2ex{
	A\ar[ddddd]_{\alpha} \ar[rrrrr]^{f}\ar[dr]^{i}&&&&& B\ar[ddrr]^{j} \ar[ddddd]^{\beta} \ar[rrrrr]^{g} &&&&& C \ar[ddddd]^{\gamma}\\
	&P \ar[rrrru]^{p_2}\ar[dr]_{p_1}\ar@{}[drrrr]|{\tiny\fbox{5}}&&&& &&&&&\\
	&&Q \ar[rrrrr]^{q_2}\ar[dddll]^{q_1}\ar@{}[dddrrr]|{\tiny\fbox{4}}&&& &&R \ar[uurrr]^{r_2}\ar[dddll]^{r_1}\ar@{}[dddrrr]|{\tiny\fbox{3}}&&& \\
	&&&&& &&&&& \\
	&&&&& &&&&& \\
	A'\ar[rrrrr]_{f'} &&&&& B' \ar[rrrrr]_{g'} &&&&& C' 
} 
\]
where $\fbox{3}$, $\fbox{4}$ and $\fbox{5}$ are pullbacks, and $i$ and $j$ the
unique morphisms making the diagram commute. 
To prove (a), suppose $\fbox{1}$ and $\fbox{2}$ are feeble pullbacks.
It follows that $i$ and $j$ are regular
epimorphisms, and hence so is $p_1$. Since regular epimorphisms
are closed under composition in a regular category,
it follows that $p_1i$ is also a regular epimorphism. To prove (b), suppose
that $j$ is an isomorphism and note that since $p_1$ is then also an
isomorphism, $p_1i$ is a regular epimorphism if and only if $i$ is. 
On the other hand to prove (c), suppose that
$\fbox{1}\fbox{2}$ is a feeble pullback and $f'$ is a 
 regular epimorphism. It follows that that $q_2$ and $p_1i$ 
are regular epimorphisms. This means that $q_2p_1i=jp_2i$ is a
 regular
epimorphism and hence $j$ is a regular epimorphism as desired.
\end{proof}
Since the coproduct of regular epimorphisms is a regular epimorphism we have:
\begin{lemma}\label{feeble_stable}
Let $\C$ be a finitely complete category with binary coproducts. 
If pullbacks are stable under binary coproduct in $\C$, then so are feeble pullbacks.\qed
\end{lemma}
For a binary relation $(R,r_1,r_2)$ on object $A$ of a category $\C$, we write
$(\bar R,\bar r_1,\bar r_2)$ for its reflexive closure. We have:
\begin{proposition}\label{proposition:7}
	Suppose that $\C$ is a regular category satisfying Condition \ref{conditions}.
       Let $m:S\to A$ be a monomorphism, and $(R,r_1,r_2)$ an equivalence
        relation on $S$ in $\C$. The reflexive closure of $(R,mr_1,mr_2)$ is
       an equivalence relation.
\end{proposition}
\begin{proof}
Note that the reflexive closure $(\bar R,\bar r_1,\bar r_2)$ of $R$ can be constructed as the join of 
$\langle mr_1,mr_2\rangle : R\to A\times A$ and $\langle 1_A,1_A\rangle: A\to A\times A$.
As recalled above, the morphism $\langle \bar r_1,\bar r_2 \rangle$, therefore forms part of the
(regular epi,mono)-factorization of $[\langle mr_1,mr_2\rangle,\langle 1_A,1_A\rangle]$
displayed in the diagram
\[
\xymatrix{
R+A \ar[dr]_{[\langle mr_1,mr_2\rangle,\langle 1_A,1_A\rangle]\ \ }\ar[rr]^{e} && \bar R \ar[dl]^{\langle \bar r_1, \bar r_2\rangle}\\
& A\times A.
}
\]
The relation $(\bar R, \bar r_1, \bar r_2)$ is easily shown to be symmetric,
and hence we only need to show that it is transitive.
Noting that $[\langle mr_1,mr_2\rangle,\langle 1_A,1_A\rangle] = \langle [mr_1,1_A],[mr_2,1_A] \rangle$
we see that $\bar r_i e = [mr_i,1_A]$. Forming the pullback
\[
\xymatrix{
R\times_{\langle r_2,r_1\rangle}R\ar[d]_{\pi_1}\ar[r]^-{\pi_2} & R\ar[d]^{r_1}\\
R \ar[r]_{r_2} & S
}
\]
it follows from Condition \ref{conditions} that each of the squares in the diagram
\[
\xymatrix@C=8ex{
R\times_{\langle r_2,r_1\rangle}R+R\ar[r]^-{\pi_2+r_2}\ar[d]_{\pi_1+1_R}&R +S \ar[r]^{1_R+m} \ar[d]_{r_1+1_S} & R+A\ar[d]^{r_1+1_A}\\
R+R \ar[r]^{r_2+r_2}\ar[d]_{[1_R,1_R]}&S+S \ar[r]^{1_S+m} \ar[d]_{[1_S,1_S]} & S+A\ar[d]^{[m,1_A]}\\
R\ar[r]_{r_2} &S \ar[r]^{m} & A
}
\]
is a feeble pullbacks. It then follows by Lemma \ref{feeble_composite} that the outer morphisms form a feeble pullback. Therefore, by Lemma \ref{feeble_stable}, the left hand
square in the diagram
	\[
		\xymatrix@C=12ex{
			(R\times_{\langle r_2,r_1\rangle}R+R)+(R+A) \ar[r]^-{(\pi_2+mr_2)+1_{R+A}} \ar[d]_{[\pi_1,1_R]+[mr_1,1_A]}&(R+A)+(R+A)\ar[r]^-{[1_{R+A},1_{R+A}]} \ar[d]_{[mr_1,1_A]+[mr_1,1_A]} & R+A\ar[d]^{[mr_1,1_A]}\\
			R+A\ar[r]_{mr_2+1_A}&A+A\ar[r]_{[1_A,1_A]} & A
		}
	\]
is a feeble pullback, and hence by Lemma \ref{feeble_composite}, so is the composite.
	This means that the diagram  
	\[
		\xymatrix@C=14ex{
			(R\times_{\langle r_2,r_1\rangle}R +R)+(R+A) \ar[r]^-{[\pi_2+mr_2,1_{R+A}]} \ar[d]_{[\pi_1,1_R]+[mr_1,1_A]} & R+A\ar[d]^{[mr_1,1_A]}\\
			R+A\ar[r]_{[mr_2,1_A]} & A
		}
	\]
	is a feeble pullback. Forming the pullback
\[
\xymatrix{
\overline R\times_{\langle\overline r_2,\overline r_1\rangle}\overline R\ar[d]_{\overline\pi_1}\ar[r]^-{\overline\pi_2} & \overline R\ar[d]^{\overline r_1}\\
\overline R \ar[r]_{\overline r_2} & A
}
\]
let us write $\tilde e$ for the unique morphism $ (R\times_{\langle r_2,r_1\rangle}R +R)+(R+A) \to \bar R\times_{\langle \bar r_2,\bar r_1\rangle} \bar R$ such that $\bar \pi_1 \tilde e  = e ([\pi_1,1_R]+[mr_1,1_A])$ and  $\bar \pi_2 \tilde e = e [\pi_2+mr_2,1_{R+A}]$. Let us write $\tau:R\times_{\langle r_2,r_1\rangle} R \to R$ for the unique morphism such that $r_1 \tau = r_1\pi_1$ and $r_2\tau = r_2\pi_2$ (which exists since $R$ is transitive). 
	Easy calculations shows that the diagram
	\[
		\xymatrix@C=15ex@R=8ex{
			& A\\
			(R\times_{\langle r_2,r_1\rangle}R +R)+(R+A)
			\ar[rr]^-{[\iota_1[\tau,1_R],1_{R+A}]} 
			\ar[dr]|{[[mr_1\pi_1,mr_1],[mr_1,1_A]]}
			\ar[ur]|{[[mr_2\pi_2,mr_2],[mr_2,1_A]]}
			&& R+A \ar[dl]^{[mr_1,1_A]}\ar[ul]_{[mr_2,1_A]}\\
			& A
		}
	\]
	commutes, $\overline r_1 \overline \pi_1 \tilde e =[[mr_1\pi_1,mr_1],[mr_1,1_A]]$ and 
 $\overline r_2 \overline \pi_2 \tilde e=[[mr_2\pi_2,mr_2],[mr_2,1_A]]$.
It now follows that the diagram 
	\[
		\xymatrix{
			(R\times_{\langle r_2,r_1\rangle}R +R)+(R+A) \ar[r]^-{\tilde e} \ar[d]|{e[\iota_1[\tau,1_R],1_{R+A}]} &  \bar R\times_{\langle \bar r_2,\bar r_1\rangle} \bar R \ar[d]|{ \langle \bar r_1\pi_2,\bar r_2\pi_2\rangle} \ar@{-->}[dl]_{\tilde \tau}\\
	\bar R \ar[r]_{\langle \bar r_1,\bar r_2\rangle} & A\times A
}
\]
	(ignoring the dotted arrow) commutes, and hence since $\tilde e$ is a regular epimorphisms, and $\langle \bar r_1, \bar r_2 \rangle$ is 
	a monomorphism, that there exits a morphism $\tilde \tau$, as shown with a dotted arrow above, making the diagram commute.
\end{proof}
\begin{proposition}\label{prop:stab_reg}
Let $\C$ be an exact category satisfying Condition \ref{conditions}.
The pushout of a monomorphism along a regular epimorphism exists,
and produces a pullback in $\C$.
Monomorphisms are stable under
pushout along regular epimorphisms in $\C$.
\end{proposition}
\begin{proof}
Let $e:S\to T$ be a regular epimorphism and let $m:S\to A$ be a monomorphism.
Writing $(R,r_1,r_2)$ for the kernel pair of $e$, Proposition \ref{proposition:7} implies 
that the reflexive closure $(\bar R,\bar r_1,\bar r_2)$ of $(R,mr_1,mr_2)$ is an equivalence relation.
Letting $p:A\to B$ be the coequalizer of $\bar r_1, \bar r_2 \bar R\to A$, one easily show that
that $p$ is also the coequalizer of $mr_1,mr_2:R\to A$, and hence there exists
$n:T\to B$ such that the diagram
\[
\xymatrix{
S\ar[r]^{e}\ar[d]_{m} & T\ar[d]^{n}\\
A\ar[r]_{p} & B
}
\]
is a pushout. Next consider the reasonably commutative diagram
\begin{equation}
\label{star}
\tag{D4}
\vcenter{
\xymatrix{
R+S \ar[d]_{1_A+m} \ar@<0.5ex>@/^6ex/[rr]^{[r_1,1_S]}\ar@<-0.5ex>@/^6ex/[rr]_{[r_2,1_S]} \ar[r]^{u} & R\ar[d]^{\tilde m} \ar@<0.5ex>[r]^{r_1} \ar@<-0.5ex>[r]_{r_2} & S\ar[r]^{e}\ar[d]_{m} & T\ar[d]^{n}\\
R+A \ar@<0.5ex>@/_6ex/[rr]^{[mr_1,1_A]}\ar@<-0.5ex>@/_6ex/[rr]_{[mr_2,1_A]} \ar[r]^{v}&  \bar R \ar@<0.5ex>[r]^{\bar r_1} \ar@<-0.5ex>[r]_{\bar r_2}& A\ar[r]_{p} & B
}
}
\end{equation}
where $v$ is the regular epimorphism used in the construction of the reflexive closure.
Since for $i=1$ and $i=2$ the two squares of the left hand diagram
\[
\xymatrix@C=8ex{
R+S \ar[d]_{1_A+m} \ar@/^6ex/[rr]^{[r_i,1_S]} \ar[r]^{r_i+1_S} & S+S\ar[d]^{1+m} \ar[r]^{[1_S,1_S]}  & S\ar[d]_{m} & 
R+S \ar[d]_{1_A+m} \ar@/^6ex/[rr]^{[r_i,1_S]} \ar[r]^{u} & R\ar[d]^{\tilde m} \ar[r]^{r_i}  & S\ar[d]_{m}\\
R+A \ar@/_6ex/[rr]^{[mr_i,1_A]} \ar[r]^{r_i+1_A}&  S+A \ar[r]^{[m,1_S]}& A &
R+A \ar@/_6ex/[rr]^{[mr_i,1_A]} \ar[r]^{v}&  \bar R \ar[r]^{\bar r_i}& A
}
\]
are feeble pullbacks, it follows from Lemma \ref{feeble_composite} (a) that the outer arrows on the left, and hence on the right
(because they are the same) 
form a feeble pullbacks. Therefore, since $v$ is regular epimorphism it follows from Lemma \ref{feeble_composite} (c) that the 
right hand commutative square is a feeble pullback.
However, since $\tilde m$ is a monomorphism, this means it is a pullback by Lemma \ref{feeble_is_pullback}.
From this we conclude that the right hand square of \eqref{star} is a pullback.
\end{proof}
\begin{proposition}\label{prop:stab_cod_incl}
Let $\C$ be a category satisfying Condition \ref{conditions}.
The pushout of a monomorphism along a coproduct inclusion exists,
and produces a pullback in $\C$.
Monomorphisms are stable under
pushout along coproduct inclusions in $\C$.
\end{proposition}
\begin{proof}
Let $\iota_2: S \to A+S$ be a coproduct inclusion and $n:S\to B$ 
a monomorphism. Noting
the diagrams
\[
\xymatrix{
0 \ar[r]\ar[d] & 0 \ar[d] & S \ar[r]^{n}\ar[d]_{1_S} & B\ar[d]^{1_B}\\
A\ar[r]_{1_A} & A & S \ar[r]_{n} & B
}
\]
are pullbacks it follows that their binary coproduct 
\[
\xymatrix{
S \ar[r]^{n} \ar[d]_{\iota_2} & B\ar[d]^{\iota_2}\\
A+S \ar[r]_{1+n} & A+B
}
\]
is a pullback. It is also a pushout. Note that $1+n$ is a monomorphism since 
each of its kernel pair projections will be the coproduct of the respective
kernel pair projections of $1$ and $n$.
\end{proof}
\begin{proposition}
Let $\C$ be an exact category satisfying Condition \ref{conditions}.
The pushout of a monomorphism along an arbitrary morphism exists,
and produces a pullback in $\C$. Monomorphisms are pushout stable in $\C$.
\end{proposition}
\begin{proof}
Just note that an arbitrary morphism $f:A\to B$ factors as
\[
\xymatrix{
A \ar[r]^-{\iota_1} & A+B \ar[r]^-{[f,1]} & B,
}
\]
and successively apply Proposition \ref{prop:stab_cod_incl} and \ref{prop:stab_reg}.
\end{proof}
\begin{corollary}
Let $\C$ be an exact category satisfying Condition \ref{conditions}.
Monomorphisms are regular in $\C$ and $\C$ is balanced. If $\C$
admits pushouts of regular epimorphisms along regular epimorphisms,
then it is coregular.
\qed 
\end{corollary}
Since the cokernel pair of an arbitrary morphism can be constructed as the pushout
of its regular image along itself we have:
\begin{corollary}
An exact category satisfying Condition \ref{conditions} admits cokernel pairs of arbitrary morphisms.\qed
\end{corollary}
\begin{proposition}
Let $\C$ be an exact category satisfying Condition \ref{conditions}.
Every reflexive corelation in $\C$ is an co-effective-equivalence-relation.
The category $\C$ is coexact as soon as it admits pushouts of regular epimorphisms along
themselves.
\end{proposition}
\begin{proof}
	Suppose that $[q_1,q_2]: A+A\to Q $ is a reflexive corelation in $\C$, that is, $[q_1,q_2]$ is an epimorphism, such that there exists $e: Q\to A$, such that $e[q_1,q_2]=[1_A,1_A]$. Let $m:S\to A$ be equalizer of $q_1$ and $q_2$. 
Since each of the squares in the diagrams
\[
\xymatrix{
A+S \ar[r]^{1+m} \ar[d]_{1+m} & A+A \ar[d]^{1+q_2} & S+A \ar[r]^{m+1} \ar[d]_{m+1} & A+A \ar[d]^{q_1+1}\\
A+A \ar[r]_{1+q_1} \ar[d]_{[1,1]} & A+Q \ar[d]^{[q_1,1]} & A+A \ar[r]_{q_2+1} \ar[d]_{[1,1]} & Q+A \ar[d]^{[1,q_2]}\\
A\ar[r]_{q_1} & Q & A\ar[r]_{q_2} & Q
}
\]
is a pullback, it follows that their composites are pullbacks. It now follows that, the left
hand square in the diagram
	\[
		\xymatrix@C=14ex{
			(A+S)+(S+A) \ar[r]^-{(1_A+m)+(m+1_A)}\ar[d]^{[1_A,m]+[m,1_A]} & (A+A)+(A+A)\ar[d]^{[q_1,q_2]+[q_1,q_2]} \ar[r]^{[1,1]} & A+A\ar[d]^{[q_1,q_2]} \\
			A+A \ar[r]_{q_1+q_2}& Q+Q \ar[r]_{[1,1]}  & Q 
		}
	\]
is a pullback. Since the right hand square is a feeble pullback, it follows that the outer arrows form a feeble pullback. 
That is the diagram
	\[
		\xymatrix@C=12ex{
			(A+S)+(S+A) \ar[r]^-{[1_A+m,m+1_A]}\ar[d]_{[1_A,m]+[m,1_A]} & A+A\ar[d]^{[q_1,q_2]} \\
			A+A \ar[r]_{[q_1,q_2]} & Q 
		}
	\]
	is a feeble pullback. Since $[q_1,q_2]$ is a regular epimorphism it also the coequalizer of $[1_A,m]+[m+1_A]$ and $[1_A+m,m+1_A]$. Since the morphisms $\iota_1+\iota_2 : A+ A \to (A+S)+(S+A)$ and $\iota_1\iota_2, \iota_2\iota_1 : S \to (A+S)+(S+A)$ are jointly epimorphic and  $([1_A,m]+[m+1_A]) (\iota_1+\iota_2) = 1_{A+A}$ and $([1_A+m,m+1_A])(\iota_1+ \iota_2)=1_{A+A}$, $([1_A,m]+[m+1_A])(\iota_1\iota_2)=\iota_1 m$ and $([1_A+m,m+1_A])(\iota_1\iota_2)=\iota_2m$, and $([1_A,m]+[m+1_A])(\iota_2\iota_1)=\iota_2m$ and $([1_A+m,m+1_A])(\iota_2\iota_1)=\iota_1m$, it follows that $[q_1,q_2]$ is the coequalizer of $\iota_1m$ and $\iota_2m$. It immediately follows that $(Q,q_1,q_2)$ is the pushout of $m$ along $m$.
\end{proof}
Combining the various results we obtain:
\begin{theorem}\label{theorem:1}
Let $\C$ be an exact category satisfying Condition \ref{conditions}.
If $\C$ has finite colimits, then $\C$ is coexact.\qed
\end{theorem}
On the other hand we have:

\begin{theorem}\label{theorem:2}
If $\C$ is an exact category satisfying Condition \ref{conditions}, then
$\C$ is coprotomodular.
\end{theorem}
\begin{proof}
Suppose that $\C$ is an exact category satisfying Condition \ref{conditions}.
Consider a diagram
\[
\xymatrix@C=2ex@R=3ex{
&& A\ar[rrr]^{j'} &&& C'\\ 
A\ar[urr]^{f} \ar[rrr]^(0.7){j} &&&C \ar[urr]^{h}\\
\\
&B\ar[uul]^{\alpha} \ar[uuur]_(0.4){\alpha'}\ar[rrr] &&& 1\ar[uul]^{\sigma}\ar[uuur]_{\sigma'}\\
}
\]
in $\C$ where $\alpha$ and $\alpha'$ are monomorphisms, $f$ is an epimorphism,
$(C,j,\sigma)$ and $(C',j'\sigma')$ are pushouts of $\alpha$ and $\alpha'$ along $B\to 1$, respectively,
and $h$ is the induced morphism between these pushouts.
As explained in the introduction and dualizing, to show $\C$ is coprotomodular, it is sufficient to show that for each
diagram as above if $h$ is an isomorphism, then $f$ is an isomorphism.
%As explained in the introduction and dualizing, it is sufficient to show that for 
%a diagram
%\[
%\xymatrix@C=2ex@R=3ex{
%&& A\ar[rrr]^{j'} &&& C'\\ 
%A\ar[urr]^{f} \ar[rrr]^(0.7){j} &&&C \ar[urr]^{h}\\
%\\
%&B\ar[uul]^{\alpha} \ar[uuur]_(0.4){\alpha'}\ar[rrr] &&& 1\ar[uul]^{\sigma}\ar[uuur]_{\sigma'}\\
%}
%\]
%in which $\alpha$ and $\alpha'$ are monomorphisms, $f$ is an epimorphism, and $\sigma$ and $\sigma'$ are obtained by pushing out $\alpha$ and $\alpha'$ along $B\to 1$, respectively,
%if the induced morphism $h$ is an isomorphism, then $f$ is an isomorphism. 
To that end suppose $h$, in the previous diagram, is an isomorphism and consider the commutative diagram
\[
\xymatrix@C=2ex@R=3ex{
&& A'\ar[rrr]^(0.6){q'}\ar@/^3ex/[rrrrrr]^{j'} &&& Q'\ar[rrr]^(0.4){m'} &&& C'\\ 
A\ar[urr]^{f} \ar[rrr]_(0.75){q}\ar@/^3ex/[rrrrrr]^{j} &&& Q \ar[urr]|(0.45){\hole}^(0.6){g} \ar[rrr]_(0.7){m} &&& C \ar[urr]^{h}\\
\\
&B\ar[uul]^{\alpha} \ar[uuur]_(0.4){\alpha'}\ar[rrr]^{p}& && S \ar[uul]^{\rho} \ar[uuur]_(0.4){\rho'}\ar[rrr] &&& 1\ar[uul]^{\sigma}\ar[uuur]_{\sigma'}\\
}
\]
obtained as follows. The morphisms at the bottom of the diagram form the (regular epi,mono)-factorization of $B\to 1$, and the remaining part is obtained by pushing out $\alpha$ and $\alpha'$ along $p$ to obtain $\rho$ and $\rho'$, respectively, and inserting the morphisms induced by the universal property of these pushouts.
Note that since epimorphisms are regular epimorphisms, and monomorphism and regular epimorphisms are pushout stable,
we have that $m$ and $m'$ are monomorphisms, and $q$, $q'$ and $g$ are regular epimorphisms.
It easily follows that $g$ is an isomorphism (being both a regular epimorphism and a momomorphism).
On the other hand, it is easy to check that the kernel pair $(R,r_1,r_2)$ of $q$, is the smallest equivalence relation containing $(B\times B, \alpha\pi_1, \alpha\pi_2)$. By Proposition \ref{proposition:7}, it follows that there is a regular epimorphism $e: (B\times B) + A \to R$ making the diagram
\[
\xymatrix{
(B\times B) + A \ar[dr]|{\langle [\alpha \pi_1,1_A], [\alpha \pi_2,1_A]\rangle} \ar[rr]^{e} && R \ar[dl]^{\langle r_1,r_2\rangle}\\
	& A\times A
}
\]
	commute. Let $(T,t_1,t_2)$ be the kernel pair of $f$ and let $s:A\to T$ be the unique morphism with $t_1s=1_A=t_2s$. Since $g$ is a isomorphism it follows that $qt_1 = qt_2$ and hence there exists $u: T \to R$ such that $r_i u = t_i$. Since the right hand square
	and the outer arrows of the diagram
	\[
		\xymatrix{
			B \ar@/^3ex/[rr]^{\alpha'}\ar@{-->}[r]_{\nu}\ar[d]_{\langle 1,1\rangle} & T \ar[r]_{\tilde f}\ar[d]^{\langle t_1,t_2 \rangle} & A'\ar[d]^{\langle 1,1\rangle}\\
			B\times B \ar@/_3ex/[rr]_{\alpha'\times \alpha'}\ar[r]^{\alpha\times \alpha} & A\times A \ar[r]^{f\times f} & A'\times A'
		}
	\]
	are pullbacks it follows that the left hand square is too. Therefore, since the
	triangles and the outer arrows in the diagram
	\[
		\xymatrix{
			B \ar[r]^{\nu}\ar[d]_{\langle 1,1\rangle} & T \ar[d]^{u} \ar@/^5ex/[ddr]^{\langle t_1,t_2\rangle}\\
			B\times B \ar[r]^{e\iota_1} \ar@/_5ex/[drr]_{\alpha\times \alpha} & R\ar[dr]^{\langle r_1,r_2\rangle}\\
			& & A\times A
		}
	\]
	commute and $\langle r_1,r_2\rangle$ is a monomorphism it follows that the square
	in the previous diagram is a pullback. It follows that both squares in the diagram
	\[
		\xymatrix{
			B+A\ar[d]_{\langle 1,1\rangle + 1} \ar[r]^{\nu + s} & T+T \ar[r]^{[1,1]}\ar[d]^{u+1} & T\ar[d]^{u}\\
			(B\times B) + A \ar[r]_{e\iota_1+s} & R+ T \ar[r]_{[1,u]} & R
		}
	\]
	are pullbacks, and hence we obtain a pullback:
	\[
		\xymatrix{
			B+A \ar[d]_{\langle 1,1\rangle+1} \ar[r]^{[\nu,s]} & T \ar[d]^{u}\\
			(B\times B) + A \ar[r]_-{e} & R.
		}
	\]
	Since $\iota_1$ and $\iota_2$ are jointly epimorphic and $[\nu,s]$ is an epimorphism
	it follows that $\nu$ and $s$ are jointly epimorphic. Therefore, since $t_1\nu=\alpha=t_2\nu$
	and $t_1s=1_T=t_2s$ 
	it follows that $t_1=t_2$ and hence $f$ is an isomorphism.
\end{proof}
Recall that a Mal'tsev category is arithmetical \cite{PEDICCHIO:1996} if 
each lattice of equivalence relations on an object in it is distributive.
\begin{theorem}\label{main-theorem}
	The opposite category of a pretopos with finite colimits is arithmetical and ideally exact.
\end{theorem}
\begin{proof}
	Let $\C$ be a pretopos.
	Combining Theorems \ref{theorem:1} and \ref{theorem:2}, we need only show that the unique morphism from $0\to 1$ in 
	$\C$ is a monomorphism, and that each lattice of co-equivalence-relations on an object in $\C$ is distributive. 
	The former is well-known and follows easily from the fact that $0$ is strict. To see why, just notice that each of the
	projections of the kernel pair of any morphism with domain $0$ must be an isomorphism.
	The latter fact follows from the fact that for each object $A$ the lattice of co-equivalence-relation on $A$ is isomorphic 
	to the subobject lattice of $A$, which is distributive (since every pretopos is coherent \cite{JOHNSTONE:2002}).
\end{proof}

\section*{Acknowledgements}
The author would like to thank Graham Manuel for proposing the question of whether pretoposes are ideally exact, as well as the Stellenbosch Category Theory research group for helpful discussions on the subject of this article.

\providecommand{\bysame}{\leavevmode\hbox to3em{\hrulefill}\thinspace}
\providecommand{\MR}{\relax\ifhmode\unskip\space\fi MR }
\providecommand{\MRhref}[2]{%
  \href{http://www.ams.org/mathscinet-getitem?mr=#1}{#2}
}
\providecommand{\href}[2]{#2}

\end{document}